\documentclass[11pt]{amsart}

\usepackage[T1]{fontenc}
\usepackage{hyperref}
\usepackage{fullpage, txfonts}
\usepackage{amsmath, amsthm, amssymb, amsfonts}
\usepackage{mathrsfs}
\usepackage{enumerate}
\usepackage[bottom]{footmisc}
\usepackage{tikz-cd}
\usepackage{cite}

\numberwithin{equation}{section}

\newtheorem{theorem}{Theorem}[section]

\newtheorem{lemma}[theorem]{Lemma}
\newtheorem{cor}[theorem]{Corollary}
\newtheorem{prop}[theorem]{Proposition}

\theoremstyle{definition}
\newtheorem{definition}[theorem]{Definition} 

\theoremstyle{remark}
\newtheorem *{remark}{Remark}
\newtheorem *{quest}{Question}
\newtheorem *{warn}{WARNING}

\numberwithin{equation}{section}

\newcommand{\co}{{\mathcal O}}
\newcommand{\C}{{\mathbb C}}
\newcommand{\R}{{\mathbb R}}
\newcommand{\CP}{{\mathbb CP}}
\newcommand{\proj}{{\mathbb P}}
\newcommand{\Z}{{\mathscr Z}}

\newcommand{\RP}{{\mathbb RP}}

\title{On the twistor theory of almost-Grassmannian manifolds}
\author{Matthew Lam}
\date{}

\begin{document}
\maketitle

\begin{abstract}
    In this document we present a twistor correspondence for half-flat almost-Grassmannian structures on real and complex manifolds. 
    We provide foundational results regarding local theory in the complex setting and a global correspondence when the underlying manifold is a real Grassmannian of 2-planes. Whereas twistor constructions typically involve moduli of closed curves in a complex manifold, we utilize and expand upon the more flexible approach pioneered by LeBrun and Mason using moduli of curves-with-boundary. 
\end{abstract}

\section{Introduction}


Twistor theory provides a mechanism for translating certain differential geometric questions into the language of complex geometry. The prototypical construction is the double fibration of flag manifolds
\begin{equation}\label{dg:doublefibration}
    \begin{tikzcd}
        & \textbf{F}_{1,2} (\C^4) \arrow[swap]{dl}{\mu} \arrow{dr}{\nu} & \\
        \textbf{F}_1 (\C^4) \ & & \textbf{F}_2 (\C^4)
    \end{tikzcd}
\end{equation}
where $\mu, \nu$ are first and second factor projections respectively. This picture allows one to transform information from $\textbf{F}_1 (\C^4)$, called the \emph{twistor space}, to $\textbf{F}_2 (\C^4)$, often called the \emph{compactified Minkowski space} or simply \emph{spacetime}. The correspondence between these two spaces has a simple geometric interpretation obtained by traversing up and down diagram (\ref{dg:doublefibration}): an element $\Pi \in \textbf{F}_2 (\C^4)$ corresponds to the set of lines contained in $\Pi$, which is $\mu(\nu^{-1}(\Pi)) \cong \CP^1$, and an element $L \in \textbf{F}_1 (\C^4)$ corresponds to the set of planes containing $L$, which is $\nu(\mu^{-1}(L)) \cong \CP^2$. Using a more involved type of transform, one can interpret general holomorphic data on the twistor space in terms of differential equations on the spacetime. This was first accomplished by Roger Penrose for self-dual metrics \cite{pen76} and the zero rest-mass equation \cite{pen681}, culminating in the seminal paper of Atiyah, Hitchin, and Singer \cite{ahs} which establishes the definitive twistor theory for 4-dimensional Riemannian geometry.

\begin{theorem}[Atiyah, Hitchin, Singer]\label{thm:ahs}
    Let $X$ be an oriented 4-manifold. Then a conformal structure on $X$ defines a natural almost-complex structure on $\proj(V_-)$, the projectified bundle of (local) anti-self-dual spinors, and this almost-complex structure is integrable if and only if the conformal structure is self-dual.
\end{theorem}
This construction is moreover invertible, and therefore describes (locally) all self-dual metrics. In their theory, self-dual conformal manifolds arise as moduli spaces of compact complex curves in the twistor space. Such moduli spaces are themselves complex manifolds, and real geometries are obtained by imposing an anti-holomorphic involution on the twistor space; the curves invariant under this involution form the desired real manifold which is automatically endowed with a self-dual conformal metric. In the split-signature (++$--$) case, this anti-holomorphic involution has fixed points and divides these invariant curves into two hemispheres. It was then observed by LeBrun and Mason that, rather than utilize an anti-holomorphic involution, one might instead focus on its fixed-point set. This subtle shift in perspective leads one to consider holomorphic curves with boundary along a totally-real submanifold inside the twistor space. The advantage of this approach is that moduli of curves-with-boundary are more flexible than moduli of closed curves fixed by an involution, which in some situations are unable to capture global phenomena. Their theory led to a successful characterization of Zoll surfaces \cite{lm0207} and self-dual Zollfrei 4-manifolds \cite{lm0702}.

In this document, the geometric structure of interest is what we call an \emph{almost-Grassmannian structure}. Elsewhere these may be referred to as \emph{paraconformal structures}, or a particular type of parabolic geometry \cite{cap}. The relevant background is summarized in section \ref{aggeometry}, but for the purposes of this introduction, a $(p, q)$ almost-Grassmannian structure on a (complex) manifold $M$ is a factorization of its tangent bundle,
\begin{equation}
    TM \cong E \otimes H,
\end{equation}
where the two factors have (complex) ranks $p,q$ respectively. This terminology is motivated by the fact that a Grassmannian manifold carries a canonical such structure. 
In the particular case $p = q = 2$, an almost-Grassmannian structure on a complex manifold is equivalent to a conformal metric with a spin structure, and the factors in the tensor product are essentially spin bundles. This definition therefore generalizes the setting in which twistor theory originally met with such success. The foundations of almost-Grassmannian geometry in the complex setting were laid out in \cite{be91}, which largely serves as a template upon which we overlay the holomorphic disk techniques developed by LeBrun and Mason.

An almost-Grassmannian structure determines preferred subspaces, which are obtained by fixing an element in one of the factors and sweeping out the other.
\begin{equation}
    \{ e \otimes h ~:~ e\in E|_p \text{ free and } h \in H|_p \text{ fixed} \} \subset T_p M
\end{equation}
When these preferred subspaces are integrable, the almost-Grassmannian structure is said to be \emph{right-flat}. As an example, the canonical almost-Grassmannian structure on a Grassmannian is right-flat. The local geometry of these manifolds is examined in sections \ref{local} and \ref{curvature}, where we establish several foundational results. The main tool here is the \emph{local twistor bundle}, which is traditionally defined in a rather ad hoc manner. Using the Ward correspondence, which relates certain holomorphic vector bundles over a twistor space with vector bundles on the spacetime, we prove a more direct and geometric description of the local twistor bundle.  

\begin{theorem}
    The Ward correspondence takes the dualized jet bundle $[J^1\co(1)]^*$ to the local twistor bundle.
\end{theorem}

The local twistor bundle is useful because it precisely encodes the curvature components governing the almost-Grassmannian geometry. As a first application, we prove a characterization of locally flat almost-Grassmannian structures in terms of these curvatures and the torsion.

\begin{theorem}
    A $(p,q)$ almost-Grassmannian manifold is locally equivalent to the standard Grassmannian if and only if
    \begin{align*}
        T_{ab}{}^c = 0 &~\text{ for }~ p,q > 2, \\
        \tilde{F}_{AA'BB'}{}^{CC'} = 0 ~\text{and}~ \Psi_{ABC}{}^D = 0 &~\text{ for }~ p>2, q=2. 
    \end{align*}
\end{theorem}
In section \ref{curvature} we discuss a twistor-theoretic obstruction to equivalence with the standard Grassmannian, and in section \ref{twistorcorr} we turn our focus to real manifolds and prove an almost-Grassmannian analog of theorem (\ref{thm:ahs}).   
\begin{theorem}\label{thm:A}
    A $(p, 2)$ real almost-Grassmannian structure defines a natural distribution of complex $p + 1$ planes on the projectivized complexified spin bundle $\proj(H \otimes \C)$. This distribution is Frobenius integrable if and only if the almost-Grassmannian structure is right-flat.
\end{theorem}
 In the special case of an oriented real Grassmannian of 2-planes, $M = \widetilde{Gr}(2, \R^{p+2})$, we obtain a global twistor correspondence for right-flat deformations of the canonical almost-Grassmannian structure. The twistor space turns out to be $\CP^{p+1}$, which follows from theorem (\ref{thm:A}) along with some topological considerations. To invert the construction, we rely on a regularity theory for holomorphic disks-with-boundary developed in \cite{leb05}. It is seen that right-flat almost-Grassmannian structures on $\widetilde{Gr}(2, \R^{p+2})$ naturally arise from moduli of holomorphic disks in $\CP^{p+1}$ with boundary along a totally-real submanifold.

\begin{theorem}
    There is a one-to-one correspondence between right-flat almost-Grassmannian structures on the oriented Grassmannian $\widetilde{Gr}(2, \R^{p+2})$ and smooth embeddings $\RP^{p+1}\hookrightarrow \CP^{p+1}$, at least near the standard ones. 
\end{theorem}
In the above correspondence, two almost-Grassmannian structures are considered to be equivalent if one is a pullback of the other by some diffeomorphism, and two embeddings of $\RP^{p+1}$ are equivalent if they differ by some reparameterization and/or biholomorphism of $\CP^{p+1}$. By simply counting parameters, it follows that there is an infinite dimensional family of non-trivial right-flat deformations. 
 Finally, in section 
\ref{holonomy} we construct a class of almost-Grassmannian manifolds with special holonomy. This is accomplished by solving the valence-2 twistor equation
\begin{equation}\label{eq:twistor2}
    \nabla_{A(A'} \omega_{B'C')} = 0
\end{equation}
whose solution, under a sort of special Lagrangian condition on the embedding $\RP^{2m+1} \hookrightarrow \CP^{2m+1}$, can be made real and parallel. 
\begin{theorem}
    Let $v$ be a divergence-free vector field on $\RP^{2m+1}$ with respect to the standard metric, and denote by $J$ the standard complex structure on $\CP^{2n+1}$. Then $Jv$ determines a family of embeddings $\RP^{2m+1} \hookrightarrow \CP^{2m+1}$ whose corresponding almost-Grassmannian structures are torsion-free and have holonomy group contained in $SO(2, R) \cdot SL(2m, R)$.
\end{theorem}

\section{Almost-Grassmannian Structures}\label{aggeometry}
We begin by recalling the definition and properties of an almost-Grassmannian structure. In this section, $M$ is a complex manifold, not necessarily compact, of dimensions $n=pq$ with $p,q \geq 2$. 

\begin{definition}
A $(p,q)$ \emph{almost-Grassmannian structure} on $M$ is an isomorphism between its holomorphic tangent bundle and a tensor product,
\begin{equation}\label{eq:ag} 
    \sigma : TM \xrightarrow{\cong} \co^A \otimes \co^{A'},
\end{equation}
where the bundles $\co^A, \co^{A'}$ have ranks $p,q$ respectively. Additionally, there is chosen (local) isomrphism between their top exterior powers
\begin{equation}\label{eq:topext}
    \alpha : \wedge^p \co_A \xrightarrow{\cong} \wedge^q \co_{A'}.
\end{equation}
An \emph{almost-Grassmannian manifold} is a manifold equipped with an almost-Grassmannian structure. 
\end{definition}

When $p=q=2$, an almost-Grassmannian structure is equivalent to a conformal class of metric with a spin structure. In analogy with this 4-dimensional case, we will refer to the bundles $\co^A, \co^{A'}$ as the (un)primed spin bundles. We will further make use of abstract index notation, with lowercase indices representing tensorial quantities and uppercase indices representing spinorial ones. These indices are related by (\ref{eq:ag}); explicitly, the isomorphism is given by a section $\sigma_a{}^{AA'} \in \Gamma(T^*M \otimes \co^A \otimes \co^{A'})$, which can be used to interchange tensorial indices with spinor indices and vice versa. For instance, a vector field $X^a$ has spinor representation $X^{AA'} = \sigma_a{}^{AA'}X^a$. However, we will typically supress this notation and regard the indices $a$ and $AA'$ as completely identical. As usual, round and square brackets indicate symmetric and antisymmetric parts respectively. 

Any choice of connections on $\co^A, \co^{A'}$ determine a conection on $TM$ and thus a torsion $T_{ab}{}^c$. This is necessarily skew in $a,b$ and due to the decomposition 
\begin{equation}\label{eq:wedgedecomp}
    \wedge^2(V \otimes W) = (\wedge^2 V \oplus \odot^2 W) \oplus (\odot^2 V \oplus \wedge^2 W),
\end{equation}
there is a spinorial deomposition 
\[ 
    T_{ab}{}^c = F_{AA'BB'}{}^{CC'} + \tilde{F}_{AA'BB'}{}^{CC'}
\] 
where $F_{AA'BB'}{}^{CC'} = F_{(AB)[A'B']}{}^{CC'}, \tilde{F}_{AA'BB'}{}^{CC'} = \tilde{F}_{[AB](A'B')}{}^{CC'}$. In fact, the totally trace-free parts of $F$ and $\tilde{F}$ are independent of the original choice of connections. (A tensor is totally trace-free if all possible traces vanish.) The significance is that the totally trace-free parts of the torsion are therefore invariants of the almost-Grassmannian structure. 

\begin{definition}
    A \emph{scale} is a nonvanishing volume form on $\co^A$, i.e. a section $\epsilon \in \Gamma(\wedge^p \co_A)$. 
\end{definition}

The unprimed spin bundle is in no way preferred over the primed spin bundle; due to the isomorphism in (\ref{eq:topext}), we could equally well have chosen a volume form on $\co^{A'}$. Since (\ref{eq:topext}) is generally defined only locally, we will often specify the choice of a \emph{local scale}. Recall that in dimension 4, an almost-Grassmannian structure determines a conformal class of metrics. In this setting, a scale coincides with a particular choice of metric in that conformal class. For general almost-Grassmannian manifolds, there is no metric involved but we still have an 
 analog of the Levi-Civita connection.

\begin{theorem}[Bailey, Eastwood]
    For any scale $\epsilon$ on an almost-Grassmannian structure, there are unique connections on $\co^A, \co^{A'}$ such that the torsion quantities $F, \tilde{F}$ are totally trace-free, and
    \begin{equation}
        \nabla_a \epsilon = 0, \nabla_a \alpha(\epsilon) = 0.
    \end{equation}
\end{theorem}

A tangent vector $V^a$ is said to be \emph{null} if its spin representation is simple, i.e. if there exist $\mu^A, \nu^{A'}$ such that $V^{AA'} = \mu^A \nu^{A'}$. By fixing a primed spinor at a point and sweeping out the unprimed spinors, we obtain a null $p$-plane. Such a plane is called an $\alpha$-plane. An $\alpha$-surface is a submanifold whose tangent space at every point is an $\alpha$-plane. Similarly, fixing an unprimed spinor and sweeping through the primed ones yields a $\beta$-plane (and corresponding $\beta$-surfaces). 

\begin{definition}
    An almost-Grassmannian manifold is \emph{right-flat} if every $\alpha$-plane is tangent to some $\alpha$-surface. In this case, the space of $\alpha$-surfaces is called the \emph{twistor space}, which need not be a manifold. 
\end{definition}

\begin{theorem}[Bailey, Eastwood] \label{thm:rightflat}
    A $(p,q)$ almost-Grassmannian manifold is right-flat if and only if 
    \begin{align*}
        \tilde{F}_{AA'BB'}{}^{CC'} = 0 &~~\text{  for  }~~ p>2;\\
        \tilde{\Psi}_{A'B'C'}{}^{D'} = 0 &~~\text{  for  }~~ p = 2\text{\footnotemark}.    
    \end{align*}
    \footnotetext{The quantity $\tilde{\Psi}_{A'B'C'}{}^{D'}$ is a component of the curvature, which serves as the almost-Grassmannian analog of the anti-self-dual Weyl tensor in 4 dimensions. For a precise definition, see the appendix of \cite{be91}.}
\end{theorem}

The key example of an almost-Grassmannian manifold is the Grassmannian $Gr(q, \C^{p+q})$. It carries a canonical $(p,q)$ almost-Grassmannian structure
\[
    TGr(q, \C^{p+q}) = Hom(\gamma, \C^{p+q}/\gamma) =  (\C^{p+q}/\gamma) \otimes \gamma^*
\] 
where $\gamma$ is the tautological $q$-plane bundle. The isomorphism between top exterior powers is induced by the standard volume form on $\C^{p+q}$. In fact, the Grassmannian with its canonical almost-Grassmannian structure is right-flat, and its $\alpha$-surfaces are all diffeomorphic to $\CP^p$. We may refer this manifold as the \emph{standard Grassmannian}, or \emph{flat model}.\\

We now discuss deformations of almost-Grassmannian structures. Recall that the spin bundles come equipped with connections (unique up to choice of scale), and that these determine a connection on $TM$. By abuse of notation, we will denote all connections by the same symbols $\nabla_a$. More explicitly, 
\begin{equation}
    \nabla_a X^b = \sigma_{BB'}{}^{b}(1 \otimes \nabla a + \nabla_a \otimes 1)(\sigma_c{}^{BB'}X^c).
\end{equation}
Any other almost-Grassmannian structure is given by some $\hat{\sigma}_a{}^{AA'}$, or equivalently an isomorphism of the tangent bundle $\Phi_a{}^b = {\sigma}_a{}^{AA'} \hat{\sigma}_{AA'}{}^b$. A deformation is therefore given by a 1-parameter family of isomorphisms throgh the identity. We can write this as a power series expansion
\[
    \Phi_a{}^b = \delta_a{}^b + t \phi_a{}^b + O(t^2),
\]
and similarly, a general change in spin connections is given by 
\begin{align*}
    \hat{\nabla}_a \mu^C &= \nabla_a \mu^C + t K_{aB}^C \mu^B + O(t^2)\\
    \hat{\nabla}_a \mu^{C'} &= \nabla_a \mu^{C'} + t \tilde{K}_{aB'}{}^{C'} \mu^{B'} + O(t^2).
\end{align*}
Setting $Q_{ab}^c = K_{aB}{}^C \delta_{B'}{}^{C'} + \tilde{K}_{aB'}{}^{C'} \delta_B {}^C$, the torsion of the induced connection on $TM$ is 
\begin{equation}\label{eq:torsiondef}
    \hat{T}_{ab}{}^c = T_{ab}{}^c - t(\nabla_{[a} \phi_{b]}{}^c + Q_{[ab]}{}^c) + O(t^2).
\end{equation}
The contorsion tensors $K, \tilde{K}$ are thus determined up to scale by the condition that $\nabla_{[a} \phi_{b]}{}^c + Q_{[ab]}{}^c$ is totally trace-free, i.e. $Q$ is exactly the trace part of $\nabla_{[a} \phi_{b]}{}^c$. Equation (\ref{eq:torsiondef}) implies that the torsion is constant, up to first order, precisely when $\nabla_{[a} \phi_{b]}{}^c$ is pure trace. We can therefore characterize the linearized right-flat deformations by solutions of the equation 
\begin{equation}\label{eq:tracefree}
    \text{Trace-free part of } \{ \nabla_{[a} \phi_{b]}{}^c\} = 0.
\end{equation}

The case when $q=2$ is particularly important, because then the $F$ component of the torsion automatically vanishes. 
For the same reason, so does the part of equation (\ref{eq:tracefree}) that is skew in $A'B'$. Thus, right-flat deformations of a $(p,2)$ almost-Grassmannian structure are given by solutions to
\begin{equation}
    \text{Trace-free part of } \{ \nabla_{[A|(A'} \phi_{B')|B]}{}^{CC'}\} = 0.
\end{equation}
In this case, points in the spacetime manifold correspond via the double fibration to embedded $\CP^1$'s in the twistor space. We will refer to these embedded curves as \emph{twistor lines}. If, in addition, $p > 2$, then theorem (\ref{thm:rightflat}) states the right-flat condition is equivalent to the vanishing of $\tilde{F}$. Thus a right-flat $(p,2)$ almost-Grassmannian manifold with $p>2$ is torsion-free.

\section{Local Aspects of Almost-Grassmannian Manifolds}\label{local}



For any $(p,2)$ almost-Grassmannian manifold, the \emph{local twistor bundle} $\tau$ is defined as an extension
\begin{equation}\label{eq:localtwistor}
    0 \rightarrow \co_{A'} \rightarrow \tau \rightarrow \co^A \rightarrow 0,
\end{equation}
which is trivialized by any (local) choice of scale. In a choice of scale, a section of the local twistor bundle is thus represented by a pair of spinors $(\omega^A, \pi_{A'})$ which transforms under rescaling $\epsilon \mapsto f \epsilon$ by
\begin{equation}\label{eq:localtwistortransform}
    \begin{pmatrix}
        \hat{\omega}^A \\ 
        \hat{\pi}_{A'} 
    \end{pmatrix} = \begin{pmatrix} 
        \omega^A \\
        \pi_{A'} - \Upsilon_{AA'} \omega^A 
    \end{pmatrix}
\end{equation}
where $\Upsilon_{AA'} = f^{-1} \nabla_{AA'} f$. This bundle carries a natural $SL(p+q)$-invariant connection given by 
\begin{equation}\label{eq:twistortransport}
    D_{AA'} \begin{pmatrix}
        {\omega}^B \\ 
        {\pi}_{B'} 
    \end{pmatrix} = \begin{pmatrix}
        \nabla_{AA'}{\omega}^B + \delta_A{}^B \pi_{A'}\\ 
        \nabla_{AA'}{\pi}_{B'} - P_{AA'BB'} \omega^B 
    \end{pmatrix}
\end{equation}
where $P$ is a curvature tensor depending on $\nabla$. For a precise definition see []; the important feature is that $P$ transforms under change of scale by
\[
    P_{AA'BB'} \to P_{AA'BB'} - \nabla_{AA'}\Upsilon_{BB'} + \Upsilon_{AB'}\Upsilon_{BA'}.
\]
This connection is contrived from the twistor equation $\nabla_{A(A'}\omega_{B')} = 0$, such that parallel sections of the local twistor bundle correspond to solutions of the twistor equation. This construction is rather ad hoc, and we will first provide a twistor characterization of this bundle in the right-flat $(p,2)$ case.\\

To set this up, we must introduce the Ward correspondence \cite{war77, pr86}. 
The Ward correspondence assigns to every holomorphic vector bundle on the twistor space, which is additionally trivial along each twistor line, a vector bundle with connection on $M$ that is flat along $\alpha$-surfaces. The basic idea is that given such a holomorphic vector bundle $\mathscr{V} \to \Z$, one can define a bundle on $M$ whose fiber at a point $x$ is the space of sections $\Gamma(L_x, \mathscr{V})$. This provides  natural identification of the fibers along a single $\alpha$-surface with the fiber over the corresponding point in $\Z$, which effectively provides parallel transport over $\alpha$-surfaces. Under mild hypotheses, this corresponence is one-to-one and we can give an alternate characterization in the reverse direction: starting with a vector bundle on $M$ that is flat along $\alpha$-surfaces, there is a holomorphic vector bundle over $\Z$ whose fiber is the space of covariantly constant sections along the corresponding $\alpha$-surface. 

\begin{definition}
    An \emph{auto-parallel spinor} is a parallel section $\pi_{A'} \in \Gamma(\Sigma, \co_{A'})$ on some $\alpha$-surface $\Sigma$, such that $v^{AA'}\pi_{A'} = 0$ for any $v$ tangent to $\Sigma$.
\end{definition}

 Denote by $\co(-1)$ the bundle on $\Z$ whose fiber at a point is the space of autoparallel spinors along the corresponding $\alpha$-surface. This is in fact a line bundle, because the condition $v^{AA'}\pi_{A'} = 0$ algebraically determines $\pi_{A'}$ at any point, and this can be propagated across $\alpha$-surfaces in a consistent manner because they are simply connected. We can then consider the 1-jet bundle $J^1\co(-1)$, whose fiber at $z$ is defined to be equivalence classes of germs of sections of $\co(-1)$, where two germs are equivalent if their first derivatives at $z$ coincide. 

We will now relate the 1-jet bundle with the tangent bundle of the original vector bundle. Let $L$ denote the total space of $\co(-1) - \textbf{0}$, where \textbf{0} denotes the zero section. Then $L$ is a smooth manifold in its own right, and has a tangent bundle $TL \to L$. The space $L$ also comes with a $\C^*$ action given by rescaling the fibers, such that $L/\C^* \cong \Z$. Each choice of $\zeta \in \C^*$ has a Jacobian so that we have an induced action on $TL$. The quotient is a rank $m+2$ bundle that we denote by $\mathscr{L}$. The situation is summarized in the diagram below. 
\begin{center}
    \begin{tikzcd}
        TL\arrow{d} \arrow{r}  &[-15pt] TL/\C^* \arrow{d} &[-30pt] = &[-30pt] \mathscr{L} \arrow{d} \\
        L \arrow{r} & L / \C^* &\cong& \Z
    \end{tikzcd}
\end{center}

\begin{lemma}
    There is a natural isomorphism $J^1\co(1) \cong \mathscr{L}^* \otimes \co(1)$.
\end{lemma}

\begin{proof}
    A section of $\co(1)$ can be interpreted as a function on $L$ that is homogeneous of degree 1 along the fibers. Similarly, a section of $\mathscr{L}$ is a vector field on $L$ that is invariant under the $\C^*$ action, i.e. also homogeneous of degree 1 along the fibers. We therefore have a differentiation map $\co({\mathscr{L}}) \otimes \co(1) \to \co(1)$, which is a perfect pairing of $\co$-modules when restricted to $\co({\mathscr{L}}) \otimes J^1\co(1)$. 
\end{proof}

\begin{theorem}\label{thm:ward}
    The Ward correspondence takes the dualized jet bundle $[J^1\co(1)]^*$ to the local twistor bundle.
\end{theorem}

\begin{proof}
    According to the previous lemma, we can instead consider $\mathscr{L} \otimes \co(-1)$. We will show that the elements of this space are in the required correspondence with parallel sections of the local twistor transport. Begin with a point $(\Sigma, \pi_{A'}) \in L$ consisting of an $\alpha$-surface and an auto-parallel spinor. Now consider a 1-parameter family of such points. The derivative of this family can be represented by another pair, one component to measure how $\Sigma$ varies and the other for $\pi_{A'}$. The first component can be denoted by $J^{AA'}$, a vector field along $\Sigma$ connecting it to an infinitesimally separated one. Note that $J^{AA'}$ is only well-defined up to a tangential component, which we can eliminate by contracting with the autoparallel spinor $\pi_{A'}$. The contraction $\omega^A = J^{AA'}\pi_{A'}$ therefore exactly encodes how $\Sigma$ varies, and then $\eta_{A'} = J^{BB'} \nabla_{BB'} \pi_{A'}$ describes the derivative of $\pi_{A'}$. The derivative of our 1-parameter family is thus given by $(\omega^A, \eta_{A'})$. In the case that $\Sigma$ does not vary, so that $J$ and $\omega$ are zero, the 1-parameter family is (to first order) a simple rescaling of $\pi_{A'}$. Then the derivative is $(0, \eta_{A'})$ where $\eta_{A'}$ is proportional to $\pi_{A'}$. 

    Observe that in either case the $\C^*$ action $\pi_{A'} \to \zeta \pi_{A'}$ induces $(\omega^A, \eta_{A'}) \to (\zeta \omega^A, \zeta \eta_{A'})$. An element of $\mathscr{L}$ is then an equivalence class $[(\omega^A, \eta_{A'})]$ under this $\C^*$ action, and the twist by $\co(-1)$ effectively cancels the $\C^*$ action. We can therefore represent elements in $\mathscr{L} \otimes \co(-1)$ by pairs $(\omega^A, \pi_{A'})$. 

    All that remains is to interpret $(\omega^A, \pi_{A'})$ in terms of the geometry on $M$. Recall that $\omega^A$ is the contraction of a connecting vector field with an autoparallel spinor $\pi_{A'}$. But connecting vector fields are just tangent vectors on $\Z$, which have the following characterization: 

    \begin{lemma}[Bailey, Eastwood]
        The fibers of $T^{1,0}\Z$ are naturally isomorphic to the space of solutions $\xi$ to the equation 
        \begin{equation}
            \text{Trace-free part of } (\pi^{A'}\pi^{B'} \nabla_{BB'} \xi_{A'}{}^A) = 0
        \end{equation}
        along the corresponding $\alpha$-surfaces, where $\pi^{A'}$ generates that $\alpha$-surface. 
    \end{lemma}

    We can choose $\pi^{A'}$ in this lemma to be dual to our auto-parallel spinor $\pi_{A'}$. It then follows immediately that $\omega^A$ solves 
    \begin{equation}
        \text{Trace-free part of } (\pi^{B'} \nabla_{BB'} \omega^A) = 0.
    \end{equation}
    Using this equation, it is a straightforward calculation that $(\omega^A, \eta_{A'})$ is parallel under local twistor transport along the corresponding $\alpha$-surface. 
\end{proof}

The local twistor bundle will be a crucial tool in our study of the local geometry. To motivate our first application, recall the following theorem in conformal geometry. 

\begin{theorem}[Weyl]
    Let $(M, [g])$ be a conformal manifold of dimension at least 3. Then the Weyl tensor $W$ vanishes if and only if $M$ is locally conformally flat. 
\end{theorem}

This theorem is so useful that it is often taken as a definition. There is an almost-Grassmannian analog of this theorem, first proposed in \cite{be91} but never proven. Let us first establish a small lemma before proving the theorem. 

\begin{lemma}\label{lem:zeroset}
    Let $(\omega^A, \pi_{A'})$ be a nonzero parallel section of the local twistor bundle. Then the zero set $\{\omega^A = 0\}$ is an $\alpha$-surface. If the local twistor connection is flat, then every $alpha$-surface arises in this way.
\end{lemma}
\begin{proof}
    Let $(\omega^A, \pi_{A'})$ be parallel and $\Sigma = \{\omega^A = 0\}$, and fix a (local) scale so that we can trivialize the local twistor bundle. Then according to the first component of the local twistor connection, $\nabla_{BB'}\omega^A = -\delta_B{}^A\pi_{B'}$, so $\omega^A$ is constant in directions $X^{BB'}$ with $X^{BB'}\pi_{B'}=0$. If $(\omega^A, \pi_{A'}) = 0$ at any point, then it is zero everywhere since $D_{AA'}$ is an $SL(p+q)$ connection. Therefore at every point in $\Sigma$, we must have $\pi_{A'} \neq 0$, so each tangent space is a null $p$-plane. Furthermore if the local twistor bundle is flat, we can arrange for any pointwise value of $(\omega^A, \pi_{A'})$, and therefore every $\alpha$-plane may be realized as the tanget space of such a zero set. 
\end{proof}

We are now ready for one of the main results. 

\begin{theorem}\label{thm:flat}
    Let $M$ be an almost-Grassmannian manifold. Then $M$ is locally equivalent to the standard Grassmannian if and only if 
    \begin{align}
        T_{ab}{}^c = 0 &~\text{ for }~ p,q > 2 \label{eq:flat1}\\
        \tilde{F}_{AA'BB'}{}^{CC'} = 0 ~~\text{ and }~~ \Psi_{ABC}{}^D = 0 &~\text{ for }~ p>2, q=2 \label{eq:flat2}
    \end{align}
\end{theorem}
\begin{proof}
    First note that the given tensors are zero for the flat model, and therefore any map preserving the standard almost-Grassmannian structure must also preserve the vanishing of these tensors. This establishes one direction of the proof.  

    For the other direction, we will first show that (\ref{eq:flat1}-\ref{eq:flat2}) implies the local twistor connection is flat. The local twistor curvature is manifestly composed of curvature components of the connection on $M$, an since the local twistor connection is scale-invariant, its curvature can only involve scale-invariant curvature components. In the case that $p,q>2$, all scale-invariant curvatures are determined by the torsion, and in the case $p>2, q=2$ all but $\Psi_{ABC}{}^D$ are determined by the torsion. Thus (\ref{eq:flat1}-\ref{eq:flat2}) implies the local twistor connection is flat. 

    Now consider the Grassmannian bundle $Gr(q, \tau) \to M$. Conceptually, we are attaching a copy of the flat model to every point in $M$. This bundle carries a canonical section $\Delta$, given by 
    \begin{equation}
        \Delta(x) = \{ (\omega^A, \pi_{A'}) \in \tau_x ~|~ \omega^A = 0\},
    \end{equation}
    which is well-defined because the unnprimed coordinate is scale-invariant. In other words, the image of $\Delta$ is just the subspace $\co_{A'} \subset \tau$. 

    We will now construct a local map from $M$ to a fiber of this bundle. To that end, fix a point $O \in $ and a neighborhood $U$ of $O$. Then we can identify the space of parallel local twistor fields with the vector space $\mathbb{T} = \tau_O$. This trivializes $\tau |_U$, and therefore the associated Grassmannian bundle 
    \begin{equation}
        \iota : Gr(q, \tau)|_U \to Gr(q, \mathbb{T}) \times U.
    \end{equation}
    Lastly, let $\rho_1 : Gr(q, \mathbb{T}) \times U \to Gr(q, \mathbb{T})$ denote the first factor projection and consider the composition
    \begin{equation}
        \varphi = \rho_1 \circ \iota \circ \Delta : U \to Gr(q, \mathbb{T}).
    \end{equation}
    Concretely, this composition takes the canonical $q$-plane over a point in $M$ and parallel transports it via the local twistor connection over to $O$. We claim that this map preserves the almost-Grassmannian structure. 

    To see this, let $\Sigma$ be an $\alpha$-surface in $U$. By (\ref{lem:zeroset}), this is equivalent to some parallel local twistor field $(\omega^A, \pi_{A'})$ up to scale, and so may be regarded as a line $\ell \subset \mathbb{T}$. Similarly, any point $x$ is a $q$-plane $\Pi_x \subset \mathbb{T}$ consisting of the twistors vanishing at $x$. But now recall from our examination of the flat model that lines in a $q$-dimensional plane $\Pi \subset \mathbb{T}$ correspond exactly to the standard $\alpha$-planes passing through $\Pi \in Gr(q, \mathbb{T})$. Since $\ell \subset \Pi_x$ for all $x \in \Sigma$, it follows that the standard $\alpha$-surface corresponding to $\ell$ passes through every point of $\varphi(\Sigma)$. 
This proves that $\varphi : U \to Gr(q, \mathbb{T})$ takes $\alpha$-surfaces to standard ones. Since an almost-Grassmannian structure is determined by its null-directions, and by extension its $\alpha$-surfaces, this shows that $M$ is locally equivalent to the standard model.
\end{proof}

Theorems (\ref{thm:rightflat}) and (\ref{thm:flat}) show that the geometry of a right-flat almost-Grassmannian manifold is essentially governed by $\Psi_{ABC}{}^D$, to the extent that it captures exactly the deviation from the flat model.

\section{A Twistorial Obstruction}\label{curvature}

Theorem (\ref{thm:flat}) asserts that the curvature component $\Psi_{ABC}{}^D$ is an obstruction to the ``flattening'' of a $(p,2)$ almost-Grassmannian structure. In the twistor space, one may also consider obstructions to trivializing the neighborhood of a twistor line. In this context, trivial means that a neighborhood of the twistor line is isomorphic to a neighborhood of a linearly embedded curve in the standard projective space. We will compute this twistorial obstruction, and demonstrate a highly suggestive relation with $\Psi_{ABC}{}^D$.

Recall the notion of an infinitesimal neighborhood; if $X$ is a complex submanifold of the complex manifold $Y$ then the $n$th infinitesimal neighborhood $X^{(n)}$ of $X$ in $Y$ is the space $X$ with the augmented ring of functions 
\begin{equation}
    \co_{X^{(n)}} := \co_Y ~/~ \mathcal{I}_X^{n+1},
\end{equation} 
where $\mathcal{I}_X$ is the ideal sheaf of functions vanishing along $X$. If $X$ has codimension $k$, then this augmented ring of functions is locally of the form 
\begin{equation}\label{eq:fat}
    \co_X[\zeta_1, \ldots, \zeta_k]/(\zeta_1, \ldots, \zeta_k)^{n+1}.
\end{equation} 
\begin{definition}
    An $n$-\emph{th order fattening} is a ringed space $(X, \co_{{(n)}})$ where $\co_{{(n)}}$ locally satisfies (\ref{eq:fat}).
\end{definition}
 A fattening need not arise from a genuine embedding of $X$ into another complex manifold, but if it does, then $\mathcal{I}^2 / \mathcal{I}$ is the co-normal bundle $N^*$. We therefore adopt this as the definition of the co-normal bundle in general, and then $\mathcal{I}^n / \mathcal{{I}}^{n-1} \cong \co(\odot^n N^*)$. We also need the extended tangent bundle $\hat{T} := Der(\co_{(1)}, \co)$ which fits into the short exact sequence
\begin{equation}\label{eq:exttangent}
    0 \to T \to \hat{T} \to N \to 0.
\end{equation}
In the case that the fattening arises from an embedding $X \hookrightarrow Y$, then $\hat{T} = T_Y | _X$ is the usual restricted tangent bundle. 

Notice that an $n$th order fattening determines a collection of lower order fattenings according to 
\begin{equation}
    \co_{(m)} = \co_{(n)} / \mathcal{I}_{(n)}^{m+1},
\end{equation}
where $\mathcal{I}_{(n)} \subset \co_{(n)}$ is the ideal of nilpotents. In this case we say that $\co_{(n)}$ extends $\co_{(m)}$. It is this extension, and more specifically obstructions to such an extension, that we are interested in. The main result we will need is the following.
\begin{theorem}[Eastwood, LeBrun]
    The obstruction to finding an extension $X^{(n+1)}$ of $X^{(n)}$, for $n \geq 1$, lies in $H^2(X, \co(\hat{T} \otimes \odot^{n+1}N^*))$. If this obstruction vanishes, then $H^1(X, \co(\hat{T} \otimes \odot^{n+1}N^*))$ acts freely and transitively on the family of extensions.
\end{theorem}



Thinking back to our situation, we are interested in how the almost-Grassmannian geometry changes when the complex structure on a neighborhood of a twistor line is deformed. By restricting to infinitesimal neighborhoods, we obtain a series of fattenings as described above. This deformation can then be measured step-by-step in the cohomology groups $H^1(X, \co(\hat{T} \otimes \odot^{n}N^*))$. Note that these groups merely \emph{act} on the family of extensions; the deformation is thus not measured in any absolute sense but relative to some other extension. For our purposes we are interested in comparing with that of a projective line in projective space, i.e. a twistor line of the flat model. 

Let's begin with the first order fattening. A first order fattening is equivalent to the trivial one if the short exact sequence in (\ref{eq:exttangent})  splits. In our case of interest, $X \approx \CP^1$ with normal bundle $N = \co^A |_x \otimes \co(1)$. For notational convenience we will write $N = \co^A(1)$, with the understanding that $\co^A$ is really a fixed vector space when we restrict to $X$. Then $H^1(X, T \otimes N^*) = H^1(\CP^1, \co_A(-1)) = 0$, so any extension $\hat{T}$ automatically splits. To first order, there is therefore no deviation from the trivial fattening, and we may assume a fixed isomorphism $\hat{T} \cong \co(2) \oplus \co^A(1)$. 

For higher order fattenings we need to analyze $H^1(X, \co(\hat{T} \otimes \odot^n N^*))$ for $n \geq 2$. We have $\odot^n N^* = \co_{(A \cdots C)}(-n)$ where there are $n$ many lower indices, so
\begin{equation}\label{eq:higherfat}
    \hat{T} \otimes \odot^n N^* = \co_{(A \cdots C)}(2-n) \oplus \co_{(A \cdots C)}{}^D (1-n).
\end{equation}
From this we can immediately read off $H^1(X, \co(\hat{T} \otimes \odot^2 N^*)) = 0$ since $H^1(\CP^1, \co(m)) = 0$ for $m \geq -1$, so that second order fattenings also cannot deviate from the trivial one. However, with $n=3$ we get that 
\begin{equation}\label{eq:3rdobstruction}
    H^1(X, \co(\hat{T} \otimes \odot^3 N^*)) \cong \co_{(ABC)}{}^D \otimes H^1(X, \co(-2)).
\end{equation}
To see how 3rd order deformations $H^1(X, \co_{(3)})$ relate to the group $H^1(X, \co(\hat{T} \otimes \odot^3 N^*))$, take the short exact sequence 
\begin{equation}
    0 \to \co(\hat{T} \otimes \odot^n N^*) \to \co_{{(n)}}(\hat{T}) \to \co_{{(n-1)}}(\hat{T}) \to 0.
\end{equation}
Analysis of the long exact sequence in cohomology for increasing values of $n$ shows that 
\begin{equation}\label{eq:cech}
    H^1(X, \co_{{(3)}}(\hat{T})) \cong H^1(X, \co(\hat{T} \otimes \odot^3 N^*)).
\end{equation}
The 3rd order fattening therefore is the first obstruction to trivializing a deformation. The group $H^1(X, \co(-2))$ seen in equation (\ref{eq:3rdobstruction}) is furthermore trivialized by a choice of scale. To see this, consider the Wronskian $\wedge^2 H \to H^0(X, \Omega^1(2))$ given by 
\[
    u\wedge v \mapsto u \otimes dv - v \otimes du,
\]
where elements in $H$ are also interpreted as sections of $\co(1)$ over the corresponding twistor line. Both vector spaces are one dimensional and the map is nonzero, so it must be an isomorphism. The required trivialization then follows from Serre duality $H^0(X, \Omega^1(2)) \cong [H^1(X, \co(-2))]^*$. All together, we have thus far shown that the 3rd order fattening is the first obstruction to trivializing a complex deformation, and a choice of scale determines an isomorphism $H^1(X, \co_{{(3)}}(\hat{T})) \cong \co_{(ABC)}{}^D$. This isomorphism is made explicit in the next proposition. 

\begin{prop}
    There is a natural map $H^1(X, \co_{{(3)}}(\hat{T})) \to \co_{(ABC)}{}^D[-1]$ which is surjective.
\end{prop}
\begin{proof}
    The notation $[-1]$ denotes a weight by $\co[-1] \equiv \co_{[A'B']}$, and means that the map is linear in choice of scale. The isomorphism (\ref{eq:3rdobstruction}) is given by taking the 3rd normal derivatives of a Cech representative, and can therefore be given as a Penrose-type contour integral. We then write an element of $H^1(X, \co_{{(3)}}\hat{T})$ in the form
    \[
        F^A \left( \frac{\mu}{\pi_1} \right) \frac{\pi_1^2}{\pi_2}
    \]
    where $\pi_1, \pi_2$ are homogeneous coordinates along the twistor line, and $\mu$ indicates all the remaining homogeneous coordinates. The vector-valued function $F^A$ is holomorphic in the affine coordinates $\mu / \pi_1$
    , and the expression $\pi_1^2 / \pi_2$ is to be understood as a Cech representative of a nontrivial element in $H^1(X, \co(1))$. Now fix a scale, which by the previous discussion provides a specific multiple of $\pi_1 d\pi_2 - \pi_2 d\pi_1 \in H^0(X, \Omega^1(2))$; denote the multiple by $\xi$. Then we have the contour integral formula
    \begin{align}
        \frac{1}{2\pi i} \oint &\frac{\partial}{\partial \mu^A}\frac{\partial}{\partial \mu^B}\frac{\partial}{\partial \mu^C} F^D\left(\frac{\mu}{\pi_1} \right) \, \frac{\pi_1^2}{\pi_2} \, (\pi_1 d\pi_2 - \pi_2 d\pi_1) \cdot \xi \nonumber \\
        &= \frac{1}{2\pi i} \oint F_{ABC}{}^D \frac{\pi_1^2}{\pi_1^3 \pi_2}\,  (\pi_1 d\pi_2 - \pi_2 d\pi_1) \cdot \xi \nonumber \\
        &= \frac{1}{2\pi i} \oint F_{ABC}{}^D\,  \left(\frac{d\pi_2}{\pi_2} - \frac {d\pi_1}{\pi_1} \right)\cdot \xi \nonumber\\
        &= F_{ABC}{}^D\,  \cdot \xi
    \end{align}
    where $F_{ABC}{}^D$ is the 3rd order coefficient of $F$. In other words, if $F$ vanishes to 2nd order, then we can write
    \begin{equation}\label{eq:2ndorder}
        F^D \left( \frac{\mu}{\pi_1} \right) = F_{ABC}{}^D \frac{\mu^A \mu^B \mu^C}{\pi_1^3} + \co(\mu^4).
    \end{equation}
    From this equation it is clear that $F_{ABC}{}^D$ can be chosen arbitrarily, and since the result is proportional to the multiple $\xi$, it follows that the contour integral formula surjects onto $\co_{(BCD)}{}^A[-1]$.
    \begin{remark}
        The preceding discussion shows that the first obstruction to trivializing the neighborhood of a twistor line can be measured by an element in $\co_{(BCD)}{}^A[-1]$. This is the same space containing the curvature $\Psi_{ABC}{}^D$, leading one to ask if these are in fact the same object. This remains a question for future investigation. 
    \end{remark}
\end{proof}

\section{The Twistor Correspondence}\label{twistorcorr}

We have thus far only been concerned with complex almost-Grassmannian manifolds. In this section we discuss the twistor construction and its inverse for the \emph{real} Grassmannian $\widetilde{Gr}(2, \R^{p+2})$. Recall that in the general case of a complex $(p,2)$ almost-Grassmannian manifold $M$, there is a $\CP^1$ bundle of $\alpha$-planes $\proj(\co^{A'}) \to M$. This total space is called the \emph{correspondence space}, and there is a natural lifting of $\alpha$-surfaces defined by sending the points in a given $\alpha$-surface to their tangent $\alpha$-planes. These lifted $\alpha$-surfaces foliate the correspondence space, and the space of leaves is the twistor space denoted by $\Z$. This space is not generally Hausdorff, but if we restrict ourselves to a neighborhood in $M$, then $\Z$ will be a complex manifold. The inverse construction is supplied by the following theorem.
\begin{theorem}[Bailey, Eastwood]
    Let $\Z$ be a $p+1$ dimensional complex manifold, containing a line $L \cong \CP^1$ with normal bundle $\oplus^p \co_L(1)$. Suppose furthermore there is some line bundle $\co(1)$ on $\Z$ with $\co(1)|_L = \co_L(1)$ and $\co(-p-2) = K_\Z$. Then a neighborhood of $L$ in $\Z$ is the twistor space of a $(p,2)$ almost-Grassmannian manifold. 
\end{theorem}

In order to link the above machinery with real geometry, we recall the notion of \emph{complexification} []. A real-analytic manifold $M$ can always be embedded in some complex manifold $\C M$, which is unique up to germ equivalence. This complex manifold $\C M$ is called the complexification of $M$, and may be explicitly realized by allowing complex numbers in the power series transition data that describes $M$. 
A factorization of real bundles extends to a factorization of complex ones, and the restrictions of these complex bundles are the complexified real ones. 
Conversely, a real manifold can be recovered from its complexification as the fixed-point set of an anti-holomorphic involution. 
At first glance there is no apparent advantage to complexification, but it turns out that geometric structures on $M$ may acquire a more transparent meaning when extended to $\C M$. 

Typically, real geometries are extracted from this picture by way of an anti-holomorphic involution on $\Z$, and the $\CP^1$'s fixed by an involution will form a real $(p, 2)$ almost-Grassmannian manifold. However, instead of focusing on the involution, we may consider its fixed point set $P$. In the case that this involution acts by conjugation on each fixed $\CP^1$ (as opposed to the antipodal map), $P$ will divide the fixed $\CP^1$'s into two hemispheres. From this perspective, the moduli space of interest is (up to double cover) the space of holomorphically embedded disks with boundary along $P$, and a deformation of the associated almost-Grassmannian structure corresponds to a deformation of the fixed-point set $P \subset \Z$. The upshot of this added complexity is that moduli of disks-with-boundary are considerably more flexible than moduli of closed curves that are fixed by an involution. This approach was first carried out for Zoll surfaces \cite{lm0207} and split-signature ASD 4-manifolds \cite{lm0702}. Here we will generalize their construction for $p>2$. 

For a real-analytic almost-Grassmannian manifold $M$, we would like to simply define its twistor space as the usual twistor space of its complexification $\C M$. But as we mentioned above, $\C M$ is only defined as the germ of the embedding of $M$, and therefore, complex $\alpha$-surfaces are only well-defined along $M$ and up to germ equivalence. Defining the twistor space in a well-defined manner will therefore require significantly more care, and we will not attempt to do so in generality. Observe first that an $\alpha$-surface in a complexification $\C M$ intersects $M$ at a single point, or is the complexification of a real $\alpha$-surface. This is a consequence of the linear algebraic fact that if a complex $\alpha$-plane contains a real null line, then that complex $\alpha$-plane is the complexification of a real $\alpha$-plane.
Our twistor space will therefore be, in essence, the union of real $\alpha$-surfaces together with the complex $\alpha$-planes which are not complexified real ones.\\ 

We will now explicitly construct the twistor space for the oriented real Grassmannian $\widetilde{Gr}(2, \R^{p+2})$. The canonical almost-Grassmannian structure yields an isomorphism $T\widetilde{Gr}(2, \R^{p+2}) = E \otimes H$ with rank$(E) = p$ and $H$ is the tautological bundle of 2-planes. As in the complex setting, the space $\proj(H)$ is foliated by lifted $\alpha$-surfaces, and we call the space of leaves the \emph{real twistor space} $Z$. In this particular setting $Z$ is a smooth manifold, which is essentially a consequence of the Reeb stability theorem \cite{thu74} and the fact that $\alpha$-surfaces here are spheres. One can see concretely that the $\alpha$-surfaces correspond to lines in $\R^{p+2}$, and so the space of leaves is identified with $\RP^{p+1}$. We therefore have the double fibration of smooth manifolds
\begin{center}
\begin{tikzcd}
    & \proj(H) \arrow[swap]{dl}{\mu} \arrow{dr}{\nu} & \\
    \widetilde{Gr}(2, \R^{p+2}) & & \RP^{p+1}
\end{tikzcd}
\end{center}
In fact, this picture holds for abitrary right-flat deformations of the standard almost-Grassmannian structure. The key observation is that $\alpha$-surfaces being diffeomorphic to spheres is an open condition in the space of right-flat almost-Grassmannian structures. Then the Reeb stability theorem will again apply, and so the deformation produces a smooth family of compact manifolds. Since one member of this family (the standard one) is diffeomorphic to $\RP^{p+1}$, they all are. 
\begin{prop}
    The real twistor space of a sufficiently small right-flat deformation of $\widetilde{Gr}(2, \R^{p+2})$ is diffeomorphic to $\RP^{p+1}$. 
\end{prop}

We now step into the complex setting by complexifying the rank 2 bundle $H \to \widetilde{Gr}(2, \R^{p+2})$. This is motivated by the aforementioned notion that our complex $\alpha$-surfaces should really be complex $\alpha$-planes along the real slice. On a complexification of the Grassmannian, the spin bundle $\co^{A'}$ restricts exactly to this complexified bundle $\C \otimes H$. 
For notational convenience we will omit the restriction and identify $\co^{A'}$ as a bundle over $\widetilde{Gr}(2, \R^{p+2})$. The correspondence space $\mathscr{F} := \proj(\co^{A'})$ is then the $\CP^1$ bundle of complex $\alpha$-planes over $\widetilde{Gr}(2, \R^{p+2})$. 


The correspondence space furthermore carries a natural distribution of complex $p+1$ planes. Fix a scale $\epsilon$ and let $\nabla$ be the associated connection. This determines a parallel transport of points in $\mathscr{F}$ along curves in $\widetilde{Gr}(2, \R^{p+2})$, and therefore a horizontal distribution in $T\mathscr{F}$. By complexifying these bundles, we may then construct a distribution $\mathcal{D}_h \subset T_\C \mathscr{F}$ of horizontally lifted complex $\alpha$-planes. Let $\mathcal{V}^{0,1} \subset T_\C \mathscr{F}$ be the $(0,1)$-tangent bundle of the fibers, and set $\mathcal{D} = \mathcal{D}_h +  \mathcal{V}^{0,1}$. The following two propositions are proved using similar arguments in section 7 of [].

\begin{prop}
    The distribution $\mathcal{D}$ is involutive if and only if the almost-Grassmannian structure on $\widetilde{Gr}(2, \R^{p+2})$ is right-flat. 
\end{prop}
\begin{proof}
    Suppose that $\mathcal{D}$ is involutive and let $F \subset \mathscr{F}$ be the equatorial section of real $\alpha$-planes. Then $T_\C F$ is certainly closed under Lie brackets, and $\mathcal{D} \cap T_\C F = \mathcal{D}_h |_F$ must be as well. This is just the complexified distribution of lifted $\alpha$-planes, and so the un-complexified distribution must also be closed under Lie brackets. By the Frobenius theorem they must be integrable, and so the almost-Grassmannian structure must be right-flat. 

    For the other direction, we will first set up a coordinate representation of $\mathcal{D}$. Let $x$ denote a local coordinate on $U \subset \widetilde{Gr}(2, \R^{p+2})$, $h_j$ be generating sections for $H$, and $e_j$ generating sections for $E$. Then $\epsilon = \wedge_j e_j$ is a scale and $\zeta$ is a complex fiber coordinate according to $\zeta \mapsto [h_1 + \zeta h_2]$. Define 
    \begin{equation}
        \phi_j = \binom{p}{j}\, h_1^jh_2^{p-j} \otimes \epsilon \in \Gamma(\odot^p H \otimes \wedge^p E),
    \end{equation}
    so we have coordinates
    \begin{equation}
        (x, \zeta) \mapsto [\phi_0 + \zeta \phi_1 + \cdots + \zeta^p \phi_p]_x
    \end{equation}
    where $\mathscr{F}$ is viewed as a submanifold of $\proj(\odot^p H \otimes \wedge^p E)$. The complex $\alpha$-plane corresponding to each $(x, \zeta)$ is simply the complex span of $\{(h_1 + \zeta h_2) \otimes e_j\}_j$, which we want to horizontally lift. With a dual basis $u^{k\ell} = (h_k \otimes e_\ell)^*$, we can write the induced connection as 
    \begin{equation}
        \nabla \phi_j = \theta_j{}^i \phi_i,
    \end{equation}
    where $\theta_j{}^i$ is a connection 1-form that can be expanded in terms of the dual basis as $\theta_j{}^i = \theta_{k\ell j}{}^i u^{k \ell}$. Here we are not using abstract indices, and the $\theta_{k\ell j}{}^i$ are simply functions. The distribution is then given by 
    \begin{equation}
        \mathcal{D} = \text{span}\left\{ \mathfrak{w}_1, \ldots, \mathfrak{w}_p, \frac{\partial}{\partial \bar{\zeta}} \right\}
    \end{equation}
    where the $\mathfrak{w}_j$ are essentially horizontal lifts of $(h_1 + \zeta h_2) \otimes e_j$. Explicitly, 
    \begin{equation}
        \mathfrak{w}_j = (h_1 + \zeta h_2) \otimes e_j + Q_j(x, \zeta) \frac{\partial}{\partial {\zeta}}, ~~~Q_j(x, \zeta) = \sum_{k, \ell = 0}^p \zeta^\ell (\theta_{1j\ell}{}^k + \zeta \theta_{2j \ell}{}^k).
    \end{equation}
    These are not quite horizontal lifts of $(h_1 + \zeta h_2) \otimes e_j$ but differ by a ${\partial}/{\partial \bar{\zeta}}$ component. They are still horizontal and real when $\zeta$ is real, but the removal of the ${\partial}/{\partial \bar{\zeta}}$ component additionally makes their coefficients holomorphic (in fact, polynomial) in $\zeta$, in terms of the basis $h_i \otimes e_j, \partial/\partial\zeta$.

    Now suppose the almost-Grassmannian structure is right-flat. As above, this implies the distribution $\mathcal{D}_h |_F \subset T_\C F$ is involutive. By construction, this distribution is spanned by $\mathfrak{w}_j$, and therefore
    \begin{equation}\label{eq:zerolie}
        [\mathfrak{w}_j, \mathfrak{w}_k] \wedge \mathfrak{w}_1 \wedge \cdots \wedge \mathfrak{w}_p = 0
    \end{equation}
    for any $j,k$. But the components of $\mathfrak{w}_j$ are holomorphic in $\zeta$, and since equation (\ref{eq:zerolie}) is zero when $\zeta$ is real, it must be zero for all $\zeta$. This show the distribution spanned by $\mathfrak{w}_j$ is involutive everywhere. Futhermore, again by holomorphicity we must have that $[\partial/\partial\bar{\zeta}, \mathfrak{w}_j] = 0$. This proves that $\mathcal{D}$ is involutive on the region parameterized by our coordinates $(x, \zeta)$. In other words, the O'Neill tensor 
    \begin{align}
        A_\mathcal{D} :\mathcal{D} \times \mathcal{D} &\to T_\C \mathscr{F}\\
        (u,v) &\mapsto [u,v] \text{ mod } \mathcal{D}
    \end{align}
    vanishes on this region. The O'Neill tensor is continuous, and so must vanish on the entirety of $\mathscr{F}|_U$. The open set $U$ was arbitrary, completing the proof that $\mathcal{D}$ is involutive. 
\end{proof}
\begin{prop}
    The distribution $\mathcal{D}$ is independent of initial choice of scale.
\end{prop}
\begin{proof}
    Observe that the distribution $\mathcal{D}_h |_F$ is independent of scale, since it is defined geometrically as the distribution of tangent spaces to lifted $\alpha$-surfaces. If we choose some other scale and construct a new distribution $\hat{\mathcal{D}}$, then it must coincide with $\mathcal{D}$ along $F$. But as before, the coefficients are holomorphic in $\zeta$. Since they agree when $\zeta$ is real, they must agree for all $\zeta$. Therefore the two distributions agree on all of $\mathscr{F}$. 
\end{proof}

We can very nearly interpret $\mathcal{D}$ as the $T^{0,1}$ component of an almost-complex structure. The problem is that $\mathcal{D}$ is real along the equatorial section $\proj(H) \subset \mathscr{F}$, i.e. ${\mathcal{D}}|_{\proj(H)} = \bar{\mathcal{D}}|_{\proj(H)}$. But away from this section, we indeed have $\mathcal{D} \cap \bar{\mathcal{D}}$, and therefore an almost-complex structure which is integrable if and only if its associated almost-Grassmannian structure is right-flat. 

Since the bundle $H$ is orientable, the inclusion $\proj(H) \subset \mathscr{F}$ divides $\mathscr{F}$ into two connected components, each of which is a disk bundle over $\widetilde{Gr}(2, \R^{p+2})$. Our chosen scale determines a particular orientation on $H$, and thus an orientation on $\proj(H)$ in the sense that each fiber is an oriented circle. At the same time, each disk bundle induces an orientation on its boundary $\proj(H)$ according to the fiberwise complex structure. We may then define $\mathscr{F}_+$ as the component whose induced orientations agree. 

Now consider the quotient map $\Psi$ on $\mathscr{F}_+$ defined as the identity on the interior and $q : \proj(H) \to \RP^{p+1}$ along the boundary. The image, denoted by $\Z$, is a smooth $2p+2$-dimensional manifold. This may be seen via coordinate representation of the map near the boundary $\partial \mathscr{F}_+ \approx \RP^{p+1} \times S^p$,
\begin{align}
    \RP^{p+1} \times S^p \times [0,1) &\to \RP^{p+1} \times \R^{p+1}\\
    (p, \vec{x}, t) &\mapsto (p, t\hat{x})
\end{align}
The 4-dimensional case of the following theorem is proven in \cite{lm0702}, and from here the proof carries over without modification.
\begin{theorem}
    Let $\widetilde{Gr}(2, \R^{p+2})$ be endowed with a right-flat almost-Grassmannian structure near the standard one. Then $\Z = \Psi(\mathscr{F_+})$ obtained as above carries a unique complex structure such that $\Psi_* \mathcal{D} \subset T^{0,1} \Z$. 
\end{theorem}
The complex manifold $\Z$ is our twistor space. Observe that a right-flat deformation of the almost-Grassmannian structure corresponds to a complex deformation of $\mathscr{F}$ and thus a deformation of $\Z$. But $\CP^{p+1}$ is biholomorphically rigid, yielding the following corollary. 

\begin{cor}
    The twistor space of $\widetilde{Gr}(2, \R^{p+2})$ equipped with a right-flat almost-Grassmannian structure near the standard one is biholomorphic to $\CP^{p+1}$. 
\end{cor}
Although the twistor spaces are indistinguishable, they do come with additional data. Namely, there is a real slice $\Psi(\partial \mathscr{F}_+) \approx \RP^{p+1} \subset \CP^{p+1}$, which is deformed along with the almost-Grassmannian structure. The canonical almost-Grassmannian structure corresponds to the standard embedding $\RP^{p+1} \subset \CP^{p+1}$, and in general the real slice allows us to completely reconstruct the original almost-Grassmannian structure, as we now demonstrate. 

\begin{theorem}\label{thm:inverse}
    Let $P \subset \CP^{p+1}$ be the image of a smooth embedding $\RP^{p+1} \hookrightarrow \CP^{p+1}$ near the standard one. Then the family of embedded holomorphic disks with boundary along $P$ is a smooth manifold diffeomorphic to $\widetilde{Gr}(2, \R^{p+2})$, which carries a natural right-flat almost-Grassmannian structure. 
\end{theorem}

The regularity theory developed in \cite{leb05} guarantees that our moduli of disks forms a smooth manifold. In particular, it is a consequence of the following theorem.
\begin{theorem}[LeBrun]
    Let $(Z,P)$ denote a complex manifold $Z$ with a totally real submanifold $P$. Suppose that $(X, \partial X) \subset (Z,P)$ is an embedded curve with boundary along $P$. If $H^1(\mathbb{X}, \co(\mathcal{N})) = 0$, where $\mathbb{X}$ is the abstract double of $X$ and $\mathcal{N}$ is the normal bundle, then any small deformation $(Z', P')$ contains a $h^0(\mathbb{X}, \co(\mathcal{N}))$ real dimensional family of curves-with-boundary obtained by deforming $X$. 
\end{theorem}

For the canonical almost-Grassmannian structure, the disks come in conjugate pairs forming closed $\CP^1$'s. Each such curve must intersect the standard complex quadric $Q = \{z_0^1 + \cdots + z_{p+1}^2 = 0 \} \subset \CP^{p+1}$ at a conjugate pair of points; one per hemisphere. A deformed disk will continue to intersect $Q$ at a single point, and conversely, since the disks must foliate $\CP^{p+1} - P$, every point along $Q$ corresponds to a disk. We therefore have a smooth family of compact manifolds, so they are all diffeomorphic to one another. This establishes the required diffeomorphism. 

Now it is only a matter of producing the almost-Grassmannian structure. It will be convenient to continue identifying the parameter space of disks with $Q$ (as a real manifold). This space carries a tautological closed disk bundle $\mathscr{F}_+ \to Q$ with a map $\Psi : \mathscr{F}_+ \to \CP^{p+1}$. By construction, $\Psi$ is a diffeomorphism on the interior of $\mathscr{F}_+$. After complexifing the tangent bundles, we have a map $\Psi_* : T_\C \mathscr{F}_+ \to T_\C \CP^{p+1}$. Let $\Psi_*^{1,0}$ denote the composition of $\Psi_*$ with projection onto the holomorphic tangent space $T^{1,0} \CP^{p+1}$ and set $\mathcal{D} = \text{ker}~\Psi_*^{1,0}$. If we assume that $\Psi$ is $C^1$ close to that of the flat model, then $\Psi_*$ will also have maximal rank, so that $\mathcal{D}$ is a rank $p+1$ complex distribution on all of $\mathscr{F}_+$. The boundary $\partial \mathscr{F}_+$, which is a $2p+1$-dimensional real manifold, is mapped to $\RP^{p+1}$ and therefore 
\begin{equation}
    E := \text{ker}~ \Psi_* |_{\partial \mathscr{F}_+}
\end{equation}
has rank at least $p$. On the other hand, $\Psi$ is fiberwise holomorphic, and so $\mathcal{D}$ contains the vertical tangent space $\mathcal{V}^{0,1}$ of the fibers. Thus $(E \otimes \C) \oplus \mathcal{V}^{0,1} \subset \mathcal{D}$, and so $E$ must have rank exactly $p$.

Now form the abstract double $\mathscr{F}$ of $\mathscr{F}_+$, by taking a duplicate copy $\mathscr{F}_-$ of $\mathscr{F}_+$ and gluing them together along their boundaries. The distribution $\mathcal{D}$ extends over the double by defining it to be the conjugate $\bar{\mathcal{D}}$ on $\mathscr{F}_-$. We then have a $\CP^1$ bundle $\wp : \mathscr{F} \to Q$ with a distribution of complex $p+1$ planes. This distribution is furthermore involutive away from the equatorial section. 

\begin{prop}\label{prop:maslov}
    The evaluation of $c_1(\mathcal{D})$ on a fiber of $\wp$ is $-p-2$. 
\end{prop}
\begin{proof}
    Consider a fiber of $\mathscr{F}_+$. By extending the normal bundle of this disk in $\CP^{p+1}$ across the abstract double, we get a splitting $N = \co(\kappa_1) \oplus \cdots \oplus \co(\kappa_p)$. The sum $\kappa = \kappa_1 + \cdots + \kappa_p$ is called the Maslov index of the disk, and for the standard $\RP^{p+1} \subset \CP^{p+1}$, we have $\kappa_j = 1$ so $\kappa = p$. The Maslov index is invariant under complex deformations \cite{ms04} and so $\kappa = p$ for all $P$ near the standard $\RP^{p+1}$. The argument in \cite{lm0702} then implies $c_1(\mathcal{D})$ on a fiber of $\wp$ is $-p-2$. 
\end{proof}

Consider now the rank $p$ distribution $\mho = \mathcal{D} / \mathcal{V}^{0,1}$. By construction, $\mho$ is mapped injectively by $\wp_*$ into $T_\C Q$, and we therefore obtain a map 
\begin{align}
    \Phi : \mathscr{F} &\to Gr(p, T_\C Q) \nonumber \\
    x &\mapsto \wp_*(\mho|_x) = \wp_*(\mathcal{D}|_x)
\end{align}
This map is holomorphic on the interior of $\mathscr{F}_+$. To see this, let $\zeta$ be a holomorphic fiber coordinate on $\mathscr{F}_+$ and let $\mathfrak{w}_j$ be $p$ sections of $\mathcal{D}$ which, with $\partial / \partial \bar{\zeta}$, span $\mathcal{D}$. Then since $\mathcal{D}$ is involutive, 
\begin{equation}
    \frac{\partial}{\partial \bar{\zeta}} (\wp_* \mathfrak{w}_j) = \wp_* (\mathcal{L}_{\partial / \partial \bar{\zeta}}\mathfrak{w}_j) = \wp_* ([\partial / \partial \bar{\zeta}, \mathfrak{w}_j]) \equiv 0 ~\text{mod span}~ \{\mathfrak{w}_1, \ldots, \mathfrak{w}_p\}, 
\end{equation}
which implies that $\Phi$ is fiberwise holomorphic on the interior of $\mathscr{F}_+$. Then $\Phi$ must be fiberwise holomorphic on the interior of $\mathscr{F}_-$ as well, and by continuity must then be fiberwise holomorphic on the entirety of $\mathscr{F}$. 

Let $\iota : Gr(p, T_\C Q) \hookrightarrow \proj(\wedge^p T_\C Q)$ denote the Plucker embedding, and consider $\iota \circ \Phi$. We will complete the proof by showing that $\iota \circ \Phi$ embeds each fiber as a rational normal curve and that such an embedding defines an almost-Grassmannian structure, whose $\alpha$-planes are integrable by construction. Since $c_1(\mathcal{V}^{0,1}) = -2$ on a fiber of $\wp$, proposition (\ref{prop:maslov}) tells us that $c^1(\mho) = -p$. Note that the tautological bundle $\co(-1)$ on $\proj(\wedge^p T_\C Q)$ pulls back via $\iota$ to the tautological $p$-plane bundle over $Gr(p, T_\C Q)$, and therefore $(\iota \circ \Phi)^*\co(-1) = \wedge^p \mho$. This implies that $\iota \circ \Phi$ is fiberwise a degree $p$ map. 

\begin{prop}\label{prop:ratnormal}
    Isomorphisms $\C^{2p} \cong \C^2 \otimes \C^p$ are in one-to-one correspondence with degree $p$ rational normal curves in $Gr(p, T_\C Q) \subset \proj(\wedge^p T_\C Q)$. 
\end{prop}
\begin{proof}
    Let $\{u_1, \ldots, u_{2p}\}$, $\{v_1, v_{2}\}$, and $\{w_1, \ldots, w_{p}\}$ be bases such that the isomorphism is given by $u_j = v_1 \otimes w_j, u_{j+p} = v_2 \otimes w_j$ for $1 \leq j \leq p$. Then there is a natural map 
    \begin{align*}
        \CP^1 &\to Gr(p, T_\C Q)\\
        [z_0, z_1] &\mapsto \text{span}\{(z_0v_1 + z_1v_2) \otimes w_1, \ldots, (z_0v_1 + z_1v_2) \otimes w_p\}\\
        &= \text{span}\{z_0 u_1 + z_1u_{p+1}, \ldots, z_0 u_p + z_1u_{2p} \}.
    \end{align*}
    The image under the Plucker embedding is then 
    \begin{equation}
        (z_0^pv_1^p + z_0^{p-1}z_1 v_1^{p-1}v_2 + \cdots + z_1^p v_2^{p}) \otimes (\wedge_{j=1}^p w_j),
    \end{equation}
    which is manifestly a degree $p$ rational curve. It is in fact a rational normal curve in the subspace $\proj(\odot^p\C^2 \otimes \wedge^p \C^p)$. Conversely, all rational normal curves are projectively equivalent, and all projective subspaces of a fixed dimension are also projectively equivalent. Thus, there exists some projective linear transformation of $\proj(\wedge^p \C^{2p})$ taking this curve to any other such curve. This is a change of basis for $\C^{2p}$, providing the required isomorphism. 
\end{proof}

With proposition (\ref{prop:ratnormal}), we just need to show that our degree $p$ maps are embeddings rather than ramified covers of lower degree curves. Suppose that a fiber maps onto a curve of lower degree $p' < p$. A degree $p'$ curve lives in a $p'$ dimensional projective subspace, and therefore the $p$-planes corresponding to this curve must share a common $p-p'$ dimensional subspace. But $\mathcal{D} \cap \bar{\mathcal{D}} = 0$ away from the equator, so there cannot be a nontrivial common subspace. It follows that $p-p'=0$, which concludes the proof of theorem (\ref{thm:inverse}). 

\section{Special Holonomy}\label{holonomy}
In the final section, we provide an almost-Grassmannian analog of Pontecorvo's characterization of Kahler metrics on compact surfaces \cite{pon92}.

\begin{theorem}[Pontecorvo]\label{thm:pontecorvo}
    Let $M$ be a complex surface with an ASD Hermitian metric. The complex structure $J$ and its conjugation $-J$ define two sections of the twistor projection $Z \to M$. Denote the images of these sections by $\Sigma, \bar{\Sigma}$. Then the divisor line bundle $[\Sigma + \bar{\Sigma}]$ satisfies 
    \begin{equation}\label{eq:pontecorvo}
        [\Sigma + \bar{\Sigma}] \cong K_Z^{-1/2}
    \end{equation}
    if and only if $M$ is conformally Kahler. 
\end{theorem}

We will focus on deformations of the flat model $\widetilde{Gr}(2, \R^{2m+2})$. Although this model does not come equipped with a preferred complex structure, there is still a notion of compatibility between an almost-complex structure and a quaternionic almost-Grassmannian structure. In particular, $J$ is compatible with $TM \cong E \otimes H$ if it arises from an almost-complex structure on the rank 2 bundle $H$. Equivalently there exist local bases $e_1, \ldots, e_{2m}$ and $h_1, h_2$ of $E,H$ such that $J$ acts by 
\begin{align}
    J(e_i \otimes h_1) &= e_i \otimes h_2 \label{eq:jcompat} \\
    J(e_i \otimes h_2) &= -e_i \otimes h_1 \nonumber
\end{align}
A choice of $J$ then determines a $2m$-dimensional complex distribution $T^{1,0} \subset T_\C M$ of $\alpha$-planes. We can therefore interpret $J$ as a section of the disc bundle $\mathscr{F}_+ \to M$. Since these are genuine complex $\alpha$-planes, as opposed to complexified real ones, this section survives the blowing down map $\Phi : \mathscr{F}_+ \to \CP^{2m+1}$, and its image here is the $m$-dimensional quadric $Q$. Note, in analogy with (\ref{eq:pontecorvo}), that 
\begin{equation}
    [Q] \cong K_Z^{-1/(m+1)}.
\end{equation}

Now consider on $\C^{m+2}$ the standard volume form paired with the Euler vector field,
\begin{equation}
    \nu = \left( z_0 \frac{\partial}{\partial z_0} + \cdots + z_{2m+2} \frac{\partial}{\partial z_{2m+2}} \right) ~\lrcorner~  (dz^0 \wedge \cdots \wedge dz^{2m+2}).
\end{equation}
Since the Euler vector field has homogeneity 1, we may retard this as a weight $2m+2$ holomorphic volume form on $\CP^{2m+1}$, i.e as a section of $K \otimes \co(2m+2)$. The quadric $Q = \{q = 0\}$ given by $J$ from above then determines a meromorphic volume form 
\begin{equation}
    \Omega = \frac{\nu}{q^{m+1}}
\end{equation}
which is holomorphic away from $Q$. Whereas the setting of theorem (\ref{thm:pontecorvo}) provides a real structure that interchanges $\Sigma$ and $\bar{\Sigma}$, we will investigate the condition that $\Omega$ restricts to a real form along the real slice $P \subset \CP^{2m+1}$. 

With $J$ and the bases $E, H$ as above, define a scale $\epsilon$ by its dual $\epsilon^* = h_1 \wedge h_2 = h_1 \wedge J(h_2)$. Let $\zeta$ be a fiber coordinate of $\mathscr{F}_+$ where $\zeta \leftrightarrow h_1 + \zeta h_2$. Then by construction, $\Psi_{-1}(Q)$ corresponds to $\zeta = i$. Since $\Psi^* \Omega$ is a section of the canonical bundle, it must annihilate $\mathfrak{w}_j, \partial/\partial \zeta$. We therefore have the coordinate expression
\begin{equation}
    \Psi^* \Omega = \frac{f}{(1+\zeta^2)^{m+1}} [(h_1^* + \zeta h_2^*)^{2m} \otimes \epsilon] \wedge [d\zeta + F(\theta_j {}^i)]
\end{equation}
where $f$ is some bounded holomorphic function on $\mathscr{F}_+$ and $F$ is a function in the connection 1-forms $\theta_j{}^i$. Note that the totally real submanifold $P \subset \CP^{2m+1}$ pulls back to the boundary of $\mathscr{F}_+$, i.e. where $\zeta$ is real. If $\Omega$ is real along $P$, if follows that $F$ must be real where $\zeta$ is real. Along each $\zeta$ fiber, the reflection principle allows us to extend $f$ to all of $\C$, and then by boundedness $f$ must be constant. We may therefore regard $f$ as a function on $M$. 

Define $\hat{\Omega}$ to be the residue of $\Psi^* \Omega$ along $\Psi^{-1}(Q)$. Since $\Psi^* \Omega$ is holomorphic away from $\zeta = i$, we can choose to compute $\hat{\Omega}$ by integrating along the boundary $\partial \mathscr{F}_+$. By the preceding paragraph, $\hat{\Omega}$ is in fact real, with the explicit formula
\begin{equation}
    \hat{\Omega} = \frac{1}{2\pi i} \oint \Psi^* \Omega = f \cdot [(h_1^*)^2 + (h_2^*)^2]^m \otimes \epsilon.
\end{equation}
Here the contour integral is meant to be taken over each fiber. Since the entire expression is real, and the tensor product is over $\C$, we can assume without loss of generality that each of the terms in the right-hand side are real. We can then factor out the rescaled scale form $f\epsilon$ to obtain a symmetric rank $m$ spinor field, which we will denote by $\omega_{A' \ldots D'}$. In other words, $\hat{\Omega} = \omega_{A' \ldots D'} \otimes f \epsilon$. The construction of $\omega$ is a very concrete form of the Penrose transform for $H^1(\CP^{m+1}, \co(-2m-2))$ and will automatically satisfy the helicity $2m$ zero rest-mass equation 
\begin{equation}\label{eq:zrm}
    \nabla_{E[E}\omega_{A']\ldots D'} = 0;
\end{equation}
see [] for details. In fact, the field $\omega$ additionally solves the twistor equation.

\begin{prop}\label{prop:twistoreq}
    The spinor field $\omega_{A' \ldots D'}$ satisfies the twistor equation $\nabla_{E(} \omega_{A') \ldots D'} = 0$.
\end{prop}
\begin{proof}
    On $\mathscr{F}_+$ there is a tautological object $\pi^{A'}$. Concretely, the total space of the bundle $\nu : \co^{A'} \to M$ carries the pullback bundle $\nu^*\co^{A'}$ which has a tautological section. This has homogeneity 1, and so descends through projectivization of the fibers to yield a section of $\nu^*\co^{A'} \otimes \co(1)$. Write $h^{A'} = h_1$, where the left-hand side uses spinor notation, and the right-hand side is one of the generating sections from before. Then we can parameterize $\pi^{A'} = h^{A'} + \zeta J(h^{A'})$, and in the same notation we have
    \[
        \omega_{A'\ldots D'} = \underbrace{[h_{A'}h_{B'} + J(h_{A'})J(h_{B'})]\cdots[h_{C'}h_{D'} + J(h_{C'})J(h_{D'})]}_{m}.
    \]
    Strictly speaking, $J$ only acts on $\co^{A'}$. By abuse of notation, $J(h_{A'})$ indicates the dual of $J(h^{A'})$. We then have
    \begin{equation}
        \pi^{A'} \cdots \pi^{D'} \omega_{A' \ldots D'} = (1 + \zeta^2)^m
    \end{equation}
    as an equation on $\mathscr{F}_+$. The important feature is that the right-hand side depends only on $\zeta$ and not upon any of the coordinates on $M$. Thus, the operator $\pi^{D'} \nabla_{DD'}$ must annihilate the left-hand side above. But the tautological object commutes with $\pi^{D'} \nabla_{DD'}$, and so
    \begin{equation}
        \pi^{A'} \cdots \pi^{C'}\pi^{D'} \nabla_{DD'}\omega_{A' \ldots D'} = 0.
    \end{equation}
    Since $\pi^{A'}$ sweeps out all possible spinors up to scale, we conclude that the primed symmetric part of $\nabla_{DD'}\omega_{A' \ldots D'}$, as an object on $M$, must vanish. 
\end{proof}

\begin{theorem}\label{thm:parallel}
    Suppose that $\Omega$ as above restricts to a real $n$-form along $P$. Then the connection on $M$ has holonomy contained in $SO(2, \R) \cdot SL(2m, \R)$. 
\end{theorem}
\begin{proof}
    Recall that $\omega$ is symmetric in all of its indices, and therefore the twistor equation of lemma (\ref{prop:twistoreq}) is equivalent to the vanishing of $\nabla_{E(E'}\omega_{A') \ldots D'}$. Since $\omega$ solves (\ref{eq:zrm}), we have
    \begin{equation}
        \nabla_{EE'}\omega_{A' \ldots D'} = \nabla_{E(E'}\omega_{A') \ldots D'} + \nabla_{E[E'}\omega_{A'] \ldots D'} = 0 
    \end{equation}
    and although $\omega$ is by construction a section of $\co_{A'\ldots D'}$, it is real and must therefore really be a section of $\odot^{2m} H^*$. In fact, $\omega$ has an $m$th root given in local coordinates by $\tilde{\omega} = h_{A'}h_{B'} + J(h_{A'})J(h_{B'})$ and the arguments showing that $w$ is parallel apply equally well to the root $\tilde{\omega}$. The only additional consideration is that $\tilde{\omega}$ may only be well defined up to an $m$th root of unity, but $M$ is simply connected, so it is indeed well-defined and we may also take it to be real. 

    The existence of a parallel section of $\odot^2 H^*$ is equivalent to a reduction of holonomy to a subgroup of $SO(2, \R) \cdot SL(2m, \R)$
\end{proof}

It may seem odd that $\omega$, initially constructed to solve the zero rest-mass equation, also happens to solve the twistor equation. The picture is a little more clear seen in reverse: the twistor equation is overdetermined and therefore the existence of a solution imposes additional structure. In fact, it is a nearly algebraic consequence that solutions to the twistor equation also satisfy the zero rest-mass equation, the proof of which is a standard exercise in tensor calculus.

\begin{prop}
    Suppose that $\omega_{A'B'} \in \Gamma(\co_{(A'B')})$ satisfies $\nabla_{C(C'} \omega_{A'B')} = 0$, is real, nondegenerate, and $|\omega| = const$. Then $\omega$ also satisfies $\nabla_{C[C'} \omega_{A']B'} = 0$. 
\end{prop}

In order to make use of theorem (\ref{thm:parallel}), we need to know that there really exist nontrivial deformations of $\RP^{2m+1} \subset \CP^{2m+1}$ satisfying the necessary reality condition. The key is to recognize that $\Omega$ restricts to the standard real volume form on $\RP^{2m+1}$, and so to preserve this condition we will require deformations corresponding to divergence-free vector fields. 

In fact, such deformations are relatively easy to come by in the real analytic case. Suppose $v$ is a real analytic vector field on $\RP^{2m+1}$ which is divergence-free with respect to the standard metric. Then if $J$ denotes the ambient complex structure, $Jv$ is a normal vector which we want to flow $\RP^{2m+1}$ along. Note that $Jv + iv$ is a section of $T^{1,0}\CP^{2m+1}|_{\RP^{2m+1}}$ and by analyticity we must have an extension to some neighborhood of $\RP^{2m+1}$. We therefore get a family of biholomorphisms $\psi_t$ of this neighborhood. Furthermore, since $\mathcal{L}_{Jv + iv}\Omega = \text{div}(Jv+iv)\Omega$ and $\text{div}(iv)|_{\RP^{2m+1}} = 0$, we see that $\mathcal{L}_{Jv + iv}\Omega$ is real along $\RP^{2m+1}$. Then $\text{Im} \psi_t^*\Omega|_{\RP^{2m+1}} = 0$, showing that divergence-free vector fields yield deformations with the desired reality condition.


\end{document}